\documentclass[a4paper,12pt,twoside]{article}
\usepackage{amssymb,amsmath,amsthm,latexsym}
\usepackage{amsfonts}
	\usepackage{amsfonts}
\usepackage{graphicx}
\usepackage[pdftex,bookmarks,colorlinks=false]{hyperref}
\usepackage[font=small,labelfont=bf]{caption}
\usepackage{subcaption}
\usepackage[hmargin=1.25in,vmargin=1.25in]{geometry}

\usepackage[auth-sc-lg,affil-sl]{authblk}
\newtheorem{theorem}{Theorem}[section]

\newtheorem{corollary}[theorem] {Corollary}
\newtheorem{definition}[theorem]{Definition}

\newtheorem{lemma} [theorem]{Lemma}

\newtheorem{problem}{Problem}
\newtheorem{proposition}[theorem]{Proposition}

\title{\sc The Sparing Number of Certain Graph Powers}

\author{\sc N. K. Sudev}
\affil{\small Department of Mathematics\\ Vidya Academy of Science \& Technology\\ Thalakkottukara, Thrissur - 680501, India.\\ E-mail: {\em sudevnk@gmail.com}}

\author{\sc K. P. Chithra}
\affil{\small Naduvath Mana, Nandikkara\\ Thrissur - 680301, India.\\ E-mail: {\em chithrasudev@gmail.com}}
 
\author{\sc K. A. Germina}
\affil{\small PG \& Research Department of Mathematics\\ Mary Matha Arts \& Science College\\ Mananthavady, Waynad-670645, E-mail:{\em srgerminaka@gmail.com}}

\date{}
\begin{document}
\maketitle

\begin{abstract}
Let $\mathbb{N}_0$ be the set of all non-negative integers and $\mathcal{P}(\mathbb{N}_0)$ be its power set. Then, an integer additive set-indexer (IASI) of a given graph $G$ is an injective function $f:V(G)\rightarrow \mathcal{P}(\mathbb{N}_0)$ such that the induced function $f^+:E(G) \rightarrow \mathcal{P}(\mathbb{N}_0)$ defined by $f^+ (uv) = f(u)+ f(v)$ is also injective. An IASI $f$ is said to be a weak IASI if $|f^+(uv)|=max(|f(u)|,|f(v)|)$ for all $u,v\in V(G)$. A graph which admits a weak IASI may be called a weak IASI graph. The set-indexing number of an element of a graph $G$, a vertex or an edge, is the cardinality of its set-labels. The sparing number of a graph $G$ is the minimum number of edges with singleton set-labels, required for a graph  $G$ to admit a weak IASI.  In this paper, we study the admissibility of weak IASI by certain graph powers and their corresponding sparing numbers.
\end{abstract}
\vspace{0.075cm}


\section{Introduction}

For all  terms and definitions, not defined specifically in this paper, we refer to \cite{BM1}, \cite{FH} and \cite{DBW}. For different graph classes, we further refer to \cite{BLS}, \cite{JAG} and \cite{GCI}. Unless mentioned otherwise, all graphs considered here are simple, finite and have no isolated vertices. 

The sumset of two non-empty sets $A$ and $B$ is denoted by  $A+B$ and is defined by $A+B = \{a+b: a \in A, b \in B\}$. Using the concept of sumsets of two sets we have the following notion.
 
Let $\mathbb{N}_0$ denote the set of all non-negative integers. An {\em integer additive set-indexer} (IASI, in short) of a graph $G$ is defined in \cite{GA} as an injective function $f:V(G)\to  \mathcal{P}(\mathbb{N}_0)$ such that the induced function $f^+:E(G) \to \mathcal{P}(\mathbb{N}_0)$ defined by $f^+ (uv) = f(u)+ f(v)$ is also injective. 

The cardinality of the labeling set of an element (vertex or edge) of a graph $G$ is called the {\em set-indexing number} of that element.

\begin{lemma}\label{L-Card}
{\rm \cite{GS1}} Let $A$ and $B$ be two non-empty finite sets of non-negative integers. Then, $\max(|A|,|B|) \le |A+B|\le |A|\,|B|$. Therefore, for an integer additive set-indexer $f$ of a graph $G$, we have $max(|f(u)|, |f(v)|)\le |f^+(uv)|= |f(u)+f(v)| \le |f(u)| |f(v)|$, where $u,v\in V(G)$.
\end{lemma}

\begin{definition}{\rm
\cite{GS1} An IASI $f$ is said to be a {\em weak IASI} if $|f^+(uv)|=|f(u)+f(v)|=\max(|f(u)|,|f(v)|)$ for all $uv\in E(G)$. A graph which admits a weak IASI may be called a {\em weak IASI graph}.

A weak  IASI $f$ is said to be {\em weakly $k$-uniform IASI} if $|f^+(uv)|=k$, for all $u,v\in V(G)$ and for some positive integer $k$.}
\end{definition}

\begin{lemma}
{\rm \cite{GS1}} An IASI $f$ define on a graph $G$ is a weak IASI of $G$ if and only if, with respect to $f$, at least one end vertex of every edge of $G$ has the set-indexing number $1$.
\end{lemma} 

\begin{definition}{\rm
\cite{GS3} An element (a vertex or an edge) of graph which has the set-indexing number $1$ is called a {\em mono-indexed element} of that graph. The {\em sparing number} of a graph $G$ is defined to be the minimum number of mono-indexed edges required for $G$ to admit a weak IASI and is denoted by $\varphi(G)$.}
\end{definition}

The following are some major results on the spring number of certain graph classes, which are relevant in our present study.

\begin{theorem}\label{T-MICn}
{\rm \cite{GS3}} An odd cycle $C_n$ contains odd number of mono-indexed edges and an even cycle contains an even number of mono-indexed edges.
\end{theorem}

\begin{theorem}\label{T-SNCn}
{\rm \cite{GS3}} The sparing number  of an odd cycle   $C_n$ is $1$ and that of an even cycle is $0$.
\end{theorem}

\begin{theorem}\label{T-SNBP}
{\rm \cite{GS3}} The sparing number  of a bipartite graph is $0$.
\end{theorem}

\begin{theorem}\label{T-SNKn}
{\rm \cite{GS3}} The sparing number  of a complete graph $K_n$ is $\frac{1}{2}(n-1)(n-2)$.
\end{theorem}

\noindent Now, let us recall the definition of graph powers.

\begin{definition}{\rm 
\cite{BM1} The $r$-th power of a simple graph $G$ is the graph $G^r$ whose vertex set is $V$, two distinct vertices being adjacent in $G^r$ if and only if their distance in $G$ is at most $r$. The graph $ G^2 $ is referred to as the {\em square} of $ G $, the graph $ G^3 $ as the {\em cube} of G.}
\end{definition}

\noindent The following is an important theorem on graph powers.

\begin{theorem}\label{T-Gdiam}
{\rm \cite{EWW}} If $d$ is the diameter of a graph $G$, then $G^d$ is a complete graph.
\end{theorem}

Some studies on the sparing numbers of certain graph classes and certain graph structures have been done in \cite{GS5}, \cite{GS6} and \cite{GS7}. As a continuation to these studies, in this paper we determine the sparing number of the powers certain graph classes.

\section{Sparing Number of  Square of Some Graphs}

In this section, we estimate the sparing number of the square of certain graph classes. It is to be noted that the weak IASI $f$ which gives the minimum number of mono-indexed edges in a given graph $G$ will not induce a weak IASI for its square graph, since some of the  vertices having non-singleton set-labels will also be at a distance $2$ in $G$. Hence, interchanging the set-labels or relabeling certain vertices may be required to obtain a weak IASI for the square graph of a given graph.

First consider a path graph $P_n$ on $n$ vertices. The following theorem provides the sparing number of the square of a path $P_n$.

\begin{proposition}
Let $G$ be a path graph on $n$ vertices. Then, 
\[\varphi(G)=
\begin{cases}
\frac{1}{3}(2n-3) & \text{if}~~ n\equiv 0~(mod~3)\\
\frac{1}{3}(2n-2) & \text{if}~~ n\equiv 1~(mod~3)\\
\frac{1}{3}(2n-1) & \text{if}~~ n\equiv 2~(mod~3)
\end{cases}
\]
\end{proposition}
\begin{proof}
Let $P_m:v_1v_2v_3\ldots v_n$, where $m=n-1$. In $P_m^2$, $d(v_1)=d(v_n)=2$ and $d(v_2)=d(v_{n-1})=3$ and $d(v_r)=4$, where $3\le r \le n-2$. Hence, $|E(P_m^2)|=\frac{1}{2}\sum_{v\in V}d(v)= \frac{1}{2}[2\times 2+ 2\times 3+ 4(n-4)]= (2n-3)$. Also, for $1\le i \le n-2$, the vertices $v_i, v_{i+1}, ~\text{and}~ v_{i+2}$ form a triangle in $P_m^2$. Then, by Theorem \ref{T-Gdiam}, each of these triangles must have a mono-indexed edge. That is, among any three consecutive vertices $v_i, v_{i+1}, ~\text{and}~ v_{i+2}$ of $P_m$, two vertices must be mono-indexed. We require an IASI which makes the maximum possible number of vertices that are not mono-indexed. Hence, label $v_1$ and $v_2$ by singleton sets and $v_3$ by a non-singleton set. Since $v_4$ and $v_5$ are adjacent to $v_3$, they can be labeled only by distinct singleton sets that are not used before for labeling. Now, $v_6$ can be labeled by a non-singleton set that has not already been used. Proceeding like this the vertices which has the form $v_{3k}, 3k\le n$ can be labeled by distinct non-singleton sets and all other vertices by singleton sets. Now, we have to consider the following cases.

\noindent {\bf Case-1:} If $n\equiv ~0~(mod ~3)$, then $n=3k$. Therefore,  $v_n$ can also be labeled by a non-singleton set. Then the number of vertices that are not mono-indexed is $\frac{n}{3}$. Therefore, the number of edges that are not mono-indexed is $4(\frac{n}{3}-1)+2= \frac{1}{3}(4n-6)$. Therefore, the total number of mono-indexed edges is $(2n-3)-\frac{1}{3}(4n-6)=\frac{1}{3}(2n-3)$.

\noindent {\bf Case-2:} If $n\equiv 1~(mod ~3)$, then $n-1=3k$. Then, $v_{n-1}$ can be labeled by a non-singleton set and $v_n$ can be labeled by a singleton set. Then the number of vertices that are not mono-indexed is $\frac{n-1}{3}$. Therefore, the number of edges that are not mono-indexed is $4(\frac{(n-1)}{3}-1)+3= \frac{1}{3}(4n-7)$. Therefore, the total number of mono-indexed edges is $(2n-3)-\frac{1}{3}(4n-7)=\frac{1}{3}(2n-2)$.

\noindent {\bf Case-3:} If $n\equiv 2~(mod ~3)$, then $n-2=3k$. Then, $v_{n-2}$ can be labeled by a non-singleton set and $v_n$ and $v_{n-1}$ can be labeled by distinct singleton sets. Then the number of vertices that are not mono-indexed is $\frac{n-2}{3}$. Therefore, the number of edges that are not mono-indexed is $4(\frac{(n-2)}{3}= \frac{1}{3}(4n-8)$. Therefore, the total number of mono-indexed edges is $(2n-3)-\frac{1}{3}(4n-8)=\frac{1}{3}(2n-1)$. 
\end{proof}

Figure \ref{G-SqPath} illustrates squares of even and odd paths which admit weak IASIs. Mono-indexed edges of the graphs are represented by dotted lines. 

\begin{figure}[h!]
\centering
\includegraphics[scale=0.5]{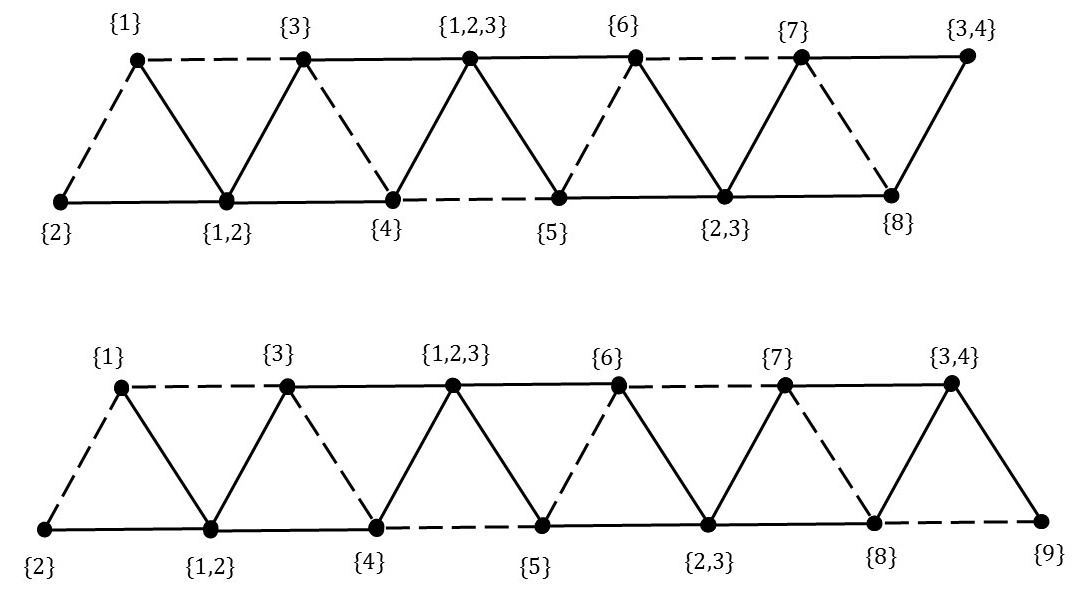}
\caption{Squares of even and odd paths which admit weak IASI}\label{G-SqPath}
\end{figure}

Next, we shall discuss the sparing number of the square of cycles. We have $C_3^2=C_3=K_3$, $C_4^2=K_4$ and $C_5^2=K_5$ and hence by Theorem \ref{T-SNKn}, their sparing numbers are $1,3 ~\text{and}~ 6$ respectively. The following theorem determines the sparing number of the square of a given cycle on $n$ vertices, for $n\ge 5$.

\begin{theorem}
Let $C_n$ be a cycle on $n$ vertices. Then, the sparing number of the square of $C_n$ is given by \[\varphi(C_n^2)=
\begin{cases}
\frac{2}{3}n & \mbox{if} ~ n\equiv 0~(mod~3)\\
\frac{2}{3}(n+2) & \mbox{if} ~ n\equiv 1~(mod~3)\\
\frac{2}{3}(n+4) & \mbox{if} ~ n\equiv 2~(mod~3). 
\end{cases}
\]
\end{theorem}
\begin{proof}
Let $C_n:v_1v_2v_3 \ldots v_nv_1$ be the given cycle on $n$ vertices. The square of $C_n$ is a $4$-regular graph. Also, $V(C_n^2)=V(C_n)$. Therefore, by the first theorem on graph theory, we have $\sum_{v\in V}d(v)=2|E|$. That is, $2|E|=4n \Rightarrow |E|=2n, n\ge 5$. 

First, label the vertex $v_1$ in $C_n^2$ by a non-singleton set. Therefore, four vertices $v_2,v_3, v_n, \text{and} v_{n-1}$ must be labeled by distinct singleton sets. Next, we can label the vertex $v_4$ by a non-singleton set, that is not already used for labeling. The vertices $v_2$ and $v_3$ have already been mono-indexed and the vertices $v_5$ and $v_6$ that are adjacent to $v_4$ in $C_n^2$ must be labeled by distinct singleton sets that are not used before for labeling. Proceeding like this, we can label all the vertices of the form $v_{3k+1}$, where $k$ is a positive integer such that $3k+1\le n-2$ (since the last vertex that remains unlabeled is $v_{n-2}$).

\noindent Here, we need to consider the following cases.

\noindent {\bf Case-1:} If $n\equiv 0~(mod~3)$, then $n-2=3k+1$ for some positive integer $k$. Then, $v_{n-2}$ can be labeled by a non-singleton set. Therefore, the number of vertices that are labeled by non-singleton sets is $\frac{n}{3}$. Since $C_n^2$ is $4$-regular, we have the number of edges that are not mono-indexed in $C_n^2$ is $\frac{4n}{3}$. Hence, the number of mono-indexed edges is $2n-\frac{4n}{3}=\frac{2n}{3}$.

\noindent {\bf Case-2:} If $n\equiv 1~(mod~3)$, then $n-2 \ne 3k+1$ for some positive integer $k$. Then, $v_{n-2}$ can not be labeled by a non-singleton set. Here $n-3=3k+1$ for some positive integer $k$. Therefore, the number of vertices that are labeled by non-singleton sets is $\frac{n-1}{3}$ and the number of edges that are not mono-indexed in $C_n^2$ is $\frac{4(n-1)}{3}$. Hence, the number of mono-indexed edges is $2n-\frac{4(n-1)}{3}=\frac{2(n+2)}{3}$.

\noindent {\bf Case-3:} If $n\equiv 2~(mod~3)$, then neither $n-2$ nor $n-3$ is equal to $3k+1$ for some positive integer $k$. Here $n-4=3k+1$ for some positive integer $k$. Therefore, the number of vertices that are labeled by non-singleton sets is $\frac{n-2}{3}$ and the number of edges that are not mono-indexed in $C_n^2$ is $\frac{4(n-2)}{3}$. Hence, the number of mono-indexed edges is $2n-\frac{4(n-2)}{3}=\frac{2(n+4)}{3}$. 
\end{proof}

Figure \ref{G-SqCyc} illustrates the admissibility of weak IASIs by the squares of cycles. The graphs given in the figure are examples to the weak IASIs of an even cycle and an odd cycle respectively. 

\begin{figure}[h!]
\centering
\includegraphics[scale=0.55]{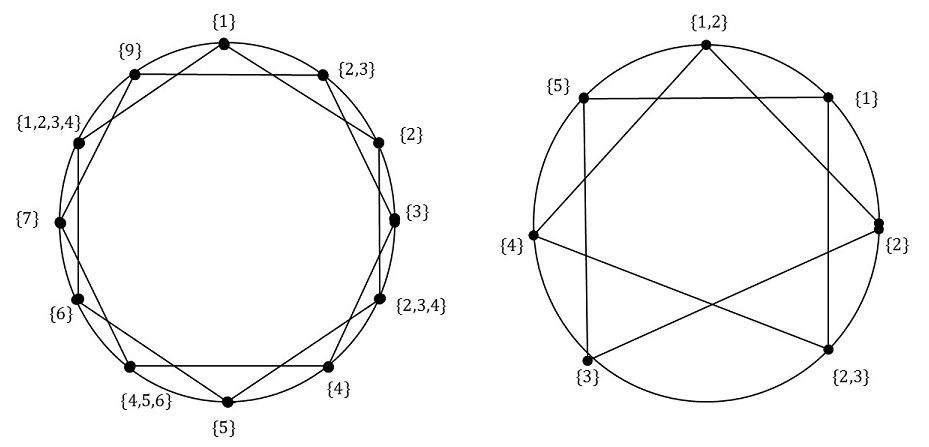}
\caption{Weak IASIs of $C_{12}^2$ and $C_7^2$. }\label{G-SqCyc}
\end{figure}

A question that arouses much interest in this context is about the sparing number of the powers of bipartite graphs. Invoking Theorem \ref{T-Gdiam}, we first verify the existence of weak IASIs for the complete bipartite graphs.

\begin{theorem}\label{T-SnKn2}
The sparing number of the square of a complete bipartite graph $K_{m,n}$ is $\frac{1}{2}(m+n-1)(m+n-1)$.
\end{theorem}
\begin{proof}
The diameter of a graph $K_{m,n}$ is $2$. Hence by Theorem \ref{T-Gdiam}, $K^2_{m,n}= K_{m+n}$.
Hence, every pair of vertices, that are not mono-indexed, are at a distance $2$. The set-labels of all these vertices, except one, must be replaced by distinct singleton sets. Therefore, by Theorem \ref{T-SNKn}, $\varphi(K^2_{m,n})= \frac{1}{2}(m+n-1)(m+n-2)$. 
\end{proof}

A {\em balanced bipartite graph} is the bipartite graph which has equal number of vertices in each of its bipartitions.

\begin{corollary}
If $G$ is a balanced complete bipartite graph on $2n$ vertices, then $\varphi(G)= (n-1)(2n-1)$
\end{corollary}
\begin{proof}
Let $G=K_{n,n}$. Then by Theorem \ref{T-SnKn2}, $\varphi(G)=\frac{1}{2}(2n-1)(2n-2)=(n-1)(2n-1)$.
\end{proof}

Let $G$ be a bipartite graph. The vertices which are at a distance $2$ are either simultaneously mono-indexed or simultaneously labeled by non-singleton sets. Therefore, in $G^2$, among any pair of vertices which are are not mono-indexed and are at a distance $2$ between them, one vertex should be relabeled by a singleton set. Hence, the sparing number of the square of a bipartite graph $G$ depends on the adjacency pattern of its vertices. Hence, the problem of finding the sparing number of bipartite graphs does not offer much scope in this context.

Now we proceed to study the admissibility of weak IASI by the squares of certain other graph classes. First, we discuss about the sparing number of {\em wheel graphs}. A wheel graph can be defined as follows.

\begin{definition}{\rm
\cite{JAG} A {\em wheel graph} is a graph defined by $W_{n+1}=C_n+K_1$. The following theorem discusses the sparing number of the square of a wheel graph $W_{n+1}$.}
\end{definition}

The sparing number of the square of a wheel graph $W_{n+1}$ is determined in the following result.

\begin{proposition}
The sparing number of the square of a wheel graph on $n+1$ vertices is $\frac{1}{2}n(n-1)$.
\end{proposition}
\begin{proof}
The diameter of a wheel graph $W_{n+1}$, for any positive integer $n\ge 3$, is $2$. Hence, by Theorem \ref{T-Gdiam}, the square of a wheel graph $W_{n+1}$ is a complete graph on $n+1$ vertices. Therefore, by Theorem \ref{T-SNKn}, the sparing number of the square graph $W_{n+1}^2$ is $\frac{1}{2}n(n-1)$.
\end{proof}

Next, we determine the sparing number of another graph class known as {\em helm graphs} which is defined as follows.  

\begin{definition}
A {\em helm graph}, denoted by $H_n$, is the graph obtained by adjoining a pendant edge to each vertex of the outer cycle $C_n$ of a wheel graph $W_{n+1}$. It has $2n+1$ vertices and $3n$ edges.
\end{definition}

\noindent The following result determines the sparing number of a helm graph.

\begin{theorem}
The sparing number of the square of a helm graph $H_n$ is $\frac{1}{2}n(n+1)$.
\end{theorem}
\begin{proof}
Let $v$ be the central vertex, $V=\{v_1,v_2,v_3,\ldots, v_n\}$ be the vertex set of the outer cycle of the corresponding wheel graph and $W=\{w_1,w_2,w_3,\ldots, w_n\}$ be the set of pendant vertices in $H_n$.

The vertex $v$ is adjacent to all the vertices in $V$ and is at distance $2$ from all the vertices in $W$. Therefore, the degree of $v$ in $H_n^2$ is $2n$. In $H_n$, for $1\le i \le n$, each $v_i$ is adjacent to two vertices $v_{i-1}$ and $v_{i+1}$ in $V$ and is adjacent to $w_i$ in $W$ and to the vertex $v$ and is at a distance $2$ from all the remaining vertices in $V$ and from the vertices $w_{i-1}$ and $w_{i+2}$ in $W$. Therefore, the degree of each $v_i\in V$ in $H_n^2$ is $n+3$. Now, in $H_n$, each vertex  $w_i$ is adjacent to the vertex $v_i$ in $V$ and is at a distance $2$ from two vertices $v_{i-1}$ and $v_{i+2}$ in $V$ and to the central vertex $v$. Hence, the degree of each $w_i\in W$ in $H_n^2$ is $4$. Therefore, the number of edges in $H_n$, $|E|=\frac{1}{2}\sum_{u\in V(H_n)}d(u)=\frac{1}{2}[2n+n(n+3)+4n]= \frac{1}{2}n(n+9)$.

It is to be noted that $W$ is an independent set in $H_n^2$ and we can label all vertices in $W$ by distinct non-singleton sets. It can be seen that there are more edges in $H_n^2$ that are not mono-indexed if we label all the vertices of $W$ by non-singleton sets than labeling possible number of vertices of $V\cup \{v\}$ by non-singleton sets. Therefore, the number of edges of $H_n^2$ which are not mono-indexed is $4n$. Therefore, the number of mono-indexed edges in $H_n^2$ is $\frac{1}{2}n(n+9)-4n=\frac{1}{2}n(n+1)$.
\end{proof}

Figure \ref{G-SqHelm} illustrates the existence of a weak IASI for the square of a helm graph.

\begin{figure}[h!]
\centering
\includegraphics[scale=0.5]{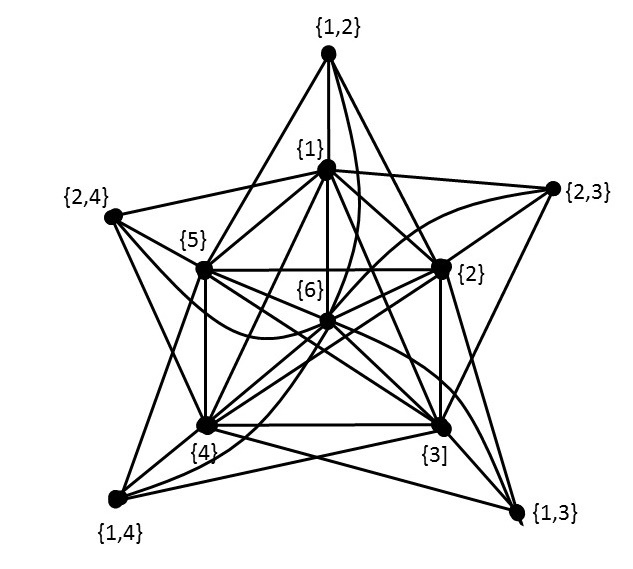}
\caption{Square of a helm graph with a weak IASI defined on it.}\label{G-SqHelm}
\end{figure}

An interesting question in this context is about the sparing number of some graph classes containing complete graphs as subgraphs. An important graph class of this kind is a complete {\em $n$-sun} which is defined as follows.

\begin{definition}{\rm
\cite{BLS} An {\em $n$-sun} or a {\em trampoline}, denoted by $S_n$,  is a chordal graph on $2n$ vertices, where $n\ge 3$, whose vertex set can be partitioned into two sets $U = \{u_1,u_2,c_3,\ldots, u_n\}$ and $W = \{w_1,w_2,w_3,\ldots, w_n\}$ such that $W$ is an independent set of $G$ and $w_j$ is adjacent to $u_i$ if and only if $j=i$ or $j=i+1~(mod ~ n)$. {\em A complete sun} is a sun $G$ where the induced subgraph $\langle U \rangle$ is complete.} 
\end{definition}

The following theorem determines the sparing number of the square of complete sun graphs.

\begin{theorem}\label{T-SqSun}
Let $G$ be the complete sun graph on $2n$ vertices. Then sparing number of $G^2$ is 
\begin{equation*}
\varphi(G^2)=
\begin{cases}
n^2+1 & \mbox{if ~ $n$ is odd}\\
\frac{n}{2}(2n-1) & \mbox{if ~ $n$ is even}.
\end{cases}
\end{equation*}
\end{theorem}
\begin{proof}
Let $G$ be a sun graph on $2n$ vertices, whose vertex set can be partitioned into two sets $U = \{u_1,u_2,c_3,\ldots, u_n\}$ and $W = \{w_1,w_2,w_3,\ldots, w_n\}$ such that $w_j$ is adjacent to $u_i$ if and only if $j=i$ or $j=i+1~(mod ~ n)$, where $W$ is an independent set and the induced subgraph $\langle U \rangle$ is complete. 

In $G$, the degree of each $u_i$ is $n+1$ and the degree of each $w_j$ is $2$. It can be seen that each vertex $w_j$ is adjacent to two vertices in $U$ and is at a distance $2$ from all other vertices in $U$. Hence, in $G^2$, each vertex $w_j$ is adjacent to all vertices in $U$ and to two vertices $w_{j-1}$ and $w_{j+1}$ (in the sense that $w_0=w_n$ and $w_{n+1}=w_1$). That is, in $G^2$, the degree of each vertex $w_j$ in $W$ is $n+2$ and the degree of each vertex $u_i$ in $U$ is $2n-1$. Therefore, $|E(G^2)|=\frac{1}{2}\sum_{v\in V}d(v)= \frac{1}{2}[n(n+2)+n(2n-1)]=\frac{1}{2}n(3n+1)$.

If we label any vertex $u_i$ by a non-singleton set, then no other vertex in $G^2$ can be labeled by non-singleton sets, as each $u_i$ is adjacent to all other vertices in $G^2$. Therefore, we label possible number of vertices in $W$ by non-singleton sets. Since $w_j$ is adjacent to $w_{j+1}$, only alternate vertices in $W$ can be labeled by non-singleton sets.

{\em Case 1:} If $n$ odd, then $\frac{1}{2}(n-1)$ vertices $W$ can be labeled by distinct non-singleton sets. Therefore, the number of edges that are not mono-indexed in $G^2$ is $\frac{1}{2}(n-1)(n+2)$. Hence, the number of mono-indexed edges in $G^2$ is  $\frac{1}{2}n(3n+1)-\frac{1}{2}(n-1)(n+2) = n^2+1$.

{\em Case 2:} If $n$ even, then $\frac{n}{2}$ vertices $W$ can be labeled by distinct non-singleton sets. Therefore, the number of edges that are not mono-indexed in $G^2$ is $\frac{1}{2}n(n+2)$. Hence, the number of mono-indexed edges in $G^2$ is  $\frac{1}{2}n(3n+1)-\frac{1}{2}n(n+2) = \frac{1}{2}n(2n-1)$.
\end{proof}

Theorem \ref{T-SqSun} is illustrated in Figure \ref{G-SqSun}. The first and second graphs in \ref{T-SqSun} are example to the weak IASIs of the square of the complete $n$-sun graphs where $n$ is odd and even respectively.

\begin{figure}[h!]
\centering
\includegraphics[scale=0.475]{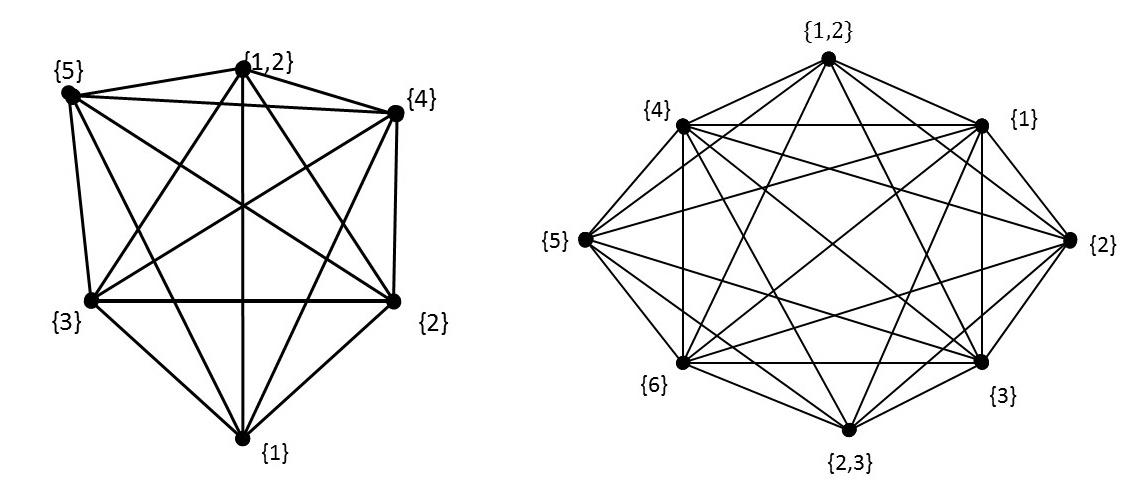}
\caption{Weak IASIs of the square of a complete $3$-sun and a complete $4$-sun.}\label{G-SqSun}
\end{figure}

Another important graph that contains a complete graph as one of its subgraph is a split graph, which is defined as follows.  

\begin{definition}{\rm
\cite{BLS} A {\em split graph} is a graph in which the vertices can be partitioned into a clique $K_r$ and an independent set $S$. A split graph is said to be a {\em complete split graph} if every vertex of the independent set $S$ is adjacent to every vertex of the the clique $K_r$ and is denoted by $K_S(r,s)$, where $r$ and $s$ are the orders of $K_r$ and $S$ respectively.}
\end{definition}

The following theorem establishes the sparing number of the square of a complete split graph. 

\begin{theorem}
Let $G=K_S(r,s)$ be a complete split graph without isolated vertices. Then, the sparing number of $G^2$ is $\frac{1}{2}[(r+s-1)(r+s-2)]$, where $r=|V(K_r)|$ and $s=|S|$.
\end{theorem}
\begin{proof}
Since $G$ has no isolated vertices, every vertex of $v_i$ $S$ is adjacent to at least one vertex $u_j$ of $K_r$. Then, $v_i$ is at a distance $2$ from all other vertices of $K_r$. Hence, in $G^2$ each vertex $v_i$ in $S$ is adjacent to all the vertices of $K_r$. Also, in $G$, two vertices of $S$ is at a distance $2$ from all other vertices of $S$. Therefore, every pair of vertices in $S$ are also adjacent in $G^2$. That is, $G^2$ is a complete graph on $r+s$ vertices. Hence, by Theorem \ref{T-SNKn}, $\varphi(G^2)=\frac{1}{2}[(r+s-1)(r+s-2)$.
\end{proof}

So far we have discussed about the sparing number of square of certain graph classes. In this context, a study about the sparing number of the higher powers of these graph classes is noteworthy. In the following section, we discuss about the sparing number of arbitrary powers of certain graph classes.

\section{Sparing Number of Arbitrary Graph Powers}

For any positive integer $n$, we know that the diameter of a complete graph $K_n$ is $1$. Hence, any power of $K_n$, denoted by $K_n^r$ is $K_n$ itself. Hence we have the following result.

\begin{proposition}
For a positive integer $r$, $\varphi(K_n^r)=\frac{1}{2}(n-1)(n-2)$.
\end{proposition}
\begin{proof}
We have $K_n^r=K_n$. Hence, $\varphi(K_n^r)=\varphi(K_n)$. Therefore, by Theorem \ref{T-SNKn}, $\varphi(K_n^r)=\frac{1}{2}(n-1)(n-2)$. 
\end{proof}

The following results discuss about the sparing numbers of the arbitrary powers of the graph classes which are discussed in Section 2.

\begin{proposition}
For a positive integer $r>1$, the sparing number of the $r$-th power of a complete bipartite graph $K_{m,n}$ is $\frac{1}{2}(m+n-1)(m+n-2)$.
\end{proposition}
\begin{proof}
Since $K_{m,n}^2=K_{m+n}$, we have $K_{m,n}^r=K_{m+n}$ for any positive integer $r\ge 2$. Therefore, $\varphi(K_{m,n}^r)=\varphi(K_{m+n})=\frac{1}{2}(m+n-1)(m+n-2)$.
\end{proof}

\begin{proposition}
Let $G$ be a split graph, without isolated vertices, that contains a clique $K_r$ and an independent set $S$ with $|S|=s$. Then, for $r\ge 3$, the sparing number of $G^r$ is $\frac{1}{2}(r+s-1)(r+s-2)$.
\end{proposition}
\begin{proof}
Since $S$ has no isolated vertices in $G$, every pair vertices of $S$ are at a distance at most $3$ among themselves. Hence, $G^3$ is a complete graph. Therefore, For any $r\ge 3$, $G^r$ is a complete graph. Hence by Theorem \ref{T-SNKn}, the sparing number of $G^r$ is $\frac{1}{2}(r+s-1)(r+s-2)$.
\end{proof}

\begin{theorem}
For a positive integer $r>2$, the sparing number of $H_n^r$ is 
\begin{equation*}
\varphi(H_n^r)=
\begin{cases}
\lfloor \frac{n}{2} \rfloor (n+3) & \text{if}~~ r=3\\
n(2n-1) & \text{if}~~ r\ge 4.
\end{cases}
\end{equation*}
\end{theorem}
\begin{proof}
Let $u$ be the central vertex, $V=\{v_1v_2v_3,\ldots, v_n\}$ be the set of vertices of the cycle $C_n$ and $W=\{w_1,w_2,w_3, \ldots,\\ w_n\}$ be the set of pendant vertices in $H_n$. In $H_n$, the central vertex $u$ is adjacent to each vertex $v_i$ of $V$ and each $v_i$ is adjacent to a vertex $w_i$ in $W$.

Since each vertex $w_i$ in $W$ is at a distance at most $3$ from $u$ as well as from all vertices of $V$, for $1\le i \le n$, and from two vertices $w_{i-1}$ and $w_{i+1}$ of $W$, the subgraph of $H_n^3$ induced by $V\cup \{u, w_{i-1},w_i, w_{i+1}\}$ is a complete graph. Hence only one vertex of this set can have a non-singleton set-label. We get minimum number of mono-indexed edges if we label possible number of vertices in $W$ by non-singleton sets. Since $w_i$ is adjacent to $w_{i-1}$ and $w_{i+1}$, only alternate vertices in $W$ can be labeled by non-singleton sets. Therefore, $\lfloor \frac{n}{2} \rfloor$ vertices in $W$ can be labeled by non-singleton sets. Therefore, since each $w_i$ is of degree $n+3$, total number of edges in $H_n^3$, that are not mono-indexed, is  $\lfloor \frac{n}{2} \rfloor (n+3)$.

The distance between any two points of a helm graph is at most $4$. Hence, $G^4$ is a complete graph. Therefore, For any $r\ge 4$, $G^r$ is a complete graph. Hence by Theorem \ref{T-SNKn}, the sparing number of $G^r$ is $n(2n-1)$.
\end{proof}

We have not determined the sparing number of arbitrary powers of paths and cycles yet. The following results discusses the sparing number of the $r$-th power of a path on $n$ vertices.

The diameter of a path $P_m$ on $n=m+1$ vertices is $m=n-1$. Therefore, by Theorem \ref{T-Gdiam}, $P_m^m= P_{n-1}^{n-1}$ is a complete graph. Hence, we need to study about the $r$-th powers of $P_{n-1}$ if $r<n-1$.

\begin{theorem}
Let $P_{n-1}$ be a path graph on $n$ vertices. Then, its spring number is  $\frac{r-1}{2(r+1)}[r(2n-1-r)+2i]$.
\end{theorem}
\begin{proof}
Let $P_m:v_1v_2v_3\ldots v_n$, where $m=n-1$. In $P_m^2$, $d(v_1)=d(v_n)=r, d(v_2)=d(v_{n-1})=r+1, \ldots, d(v_r) = d(v_{n-r+1} = r+r-1=2r-1$ and $d(v_j)=2r, ~r+1\le j \le n-r$. Hence, $\displaystyle{\sum_{v\in V(P_n)}}d(v)= 2[r+(r+1)+(r+2)+\ldots + 2r-1)]+(n-2r)2r=r(2n-1-r)$. Therefore, $|E(P_m^r)|=\frac{r}{2}(2n-1-r)$. 

It can be seen that among any $r+1$ consecutive vertices $v_i, v_{i+1}, \ldots v_{i+r}$ of $P_m$, $r$ vertices must be mono-indexed. Hence, label $v_1,v_2,\ldots, v_k$ by singleton sets and $v_{r+1}$ by a non-singleton set. Since $v_{r+2},v_{r+3}\ldots, v_{2r+1}$ are adjacent to $v_{r+1}$, they can be labeled only by distinct singleton sets that are not used before for labeling. Now, $v_{2r+2}$ can be labeled by a non-singleton set that has not already been used. Proceeding like this the vertices which has the form $v_{(r+1)k}, (r+1)k\le n$ can be labeled by distinct non-singleton sets and all other vertices by singleton sets. 

If $n\equiv ~i~(mod ~(k+1))$, then $v_{n-i}$ can also be labeled by a non-singleton set. Then the number of vertices that are not mono-indexed is $\frac{n-i}{r+1}$. Therefore, the number of edges that are not mono-indexed is $2r[\frac{(n-i)}{r+1}-1]+(r+i)= \frac{1}{r+1}[r(2n-1-r)-(r-1)i]$. Therefore, the total number of mono-indexed edges is $\frac{r}{2}(2n-1-r)-\frac{1}{r+1}[r(2n-1-r)-(r-1)i]=\frac{r-1}{2(r+1)}[r(2n-1-r)+2i]$.
\end{proof}

Figure \ref{G-kPPath} depicts the cube of a path with a weak IASI defined on it. 

\begin{figure}[h!]
\centering
\includegraphics[scale=0.45]{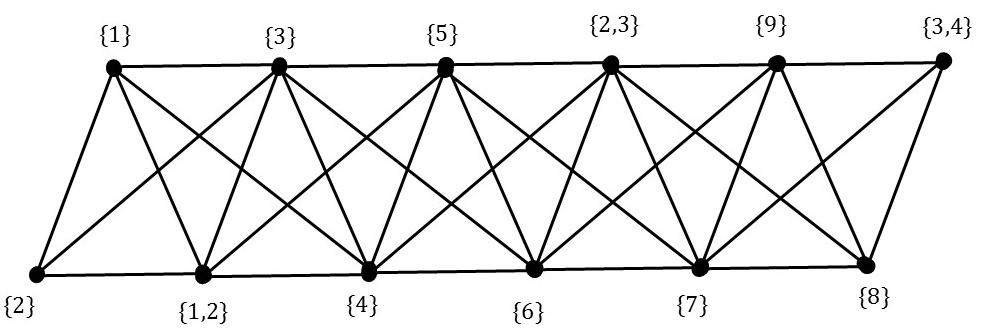}
\caption{Cubes of a path which admits a weak IASI}\label{G-kPPath}
\end{figure}

The diameter of a cycle $C_n$ is $\lfloor \frac{n}{2} \rfloor$. Therefore, by Theorem \ref{T-Gdiam}, $C_n^{\lfloor \frac{n}{2} \rfloor}$ (and higher powers) is a complete graph. Hence, we need to study about the $r$-th power of $C_n$ if $r < \lfloor \frac{n}{2} \rfloor $. The following theorem discusses about the sparing number of an arbitrary power of a cycle.

\begin{theorem}
Let $C_n$ be a cycle on $n$ vertices and let $r$ be a positive integer less than $\lfloor \frac{n}{2} \rfloor $. Then the sparing number of the the $r$-th power of $C_n$ is given by $\varphi(C_n^r)=\frac{r}{r+1}((r-1)n+2i) ~~\text{if} ~ n\equiv i~(mod~(r+1))$.
\end{theorem}
\begin{proof}
Let $C_n:v_1v_2v_3 \ldots v_nv_1$ be the given cycle on $n$ vertices. The graph $C_n^r$ is a $2r$-regular graph. Therefore,  we have $|E(C_n^r)|=\frac{1}{2}\sum_{v\in V}d(v)=rn$.  

First, label the vertex $v_1$ in $C_n^r$ by a non-singleton set. Therefore, $2r$ vertices $v_2,v_3,\ldots v_{r+1}, v_n,v_{n-1}\dots v_{n-r+1}$ can be labeled only by distinct singleton sets. Next, we can label the vertex $v_{r+2}$ by a non-singleton set, that is not already used for labeling. Since the vertices $v_2,v_3,\ldots v_{r+1}$ have already been mono-indexed, $r$ vertices $v_{r+3}, v_{r+4}, \ldots v_{2r+2}$ that are adjacent to $v_{r+2}$ in $C_n^r$ must be labeled by distinct singleton sets. Proceeding like this, we can label all the vertices of the form $v_{(r+1)k+1}$, where $k$ is a positive integer less than $\lfloor n \rfloor$, such that $(r+1)k+1\le n-r$ (since the last vertex that remains unlabeled is $v_{n-r}$).

If $n\equiv i~(mod~(k+1))$, then $n-i=(r+1)k+1$ for some positive integer $k$. Then, $v_{n-(r-i)}$ can be labeled by a non-singleton set. Therefore, the number of vertices that are labeled by non-singleton set is $\frac{n-i}{r+1}$. Since $C_n^r$ is $2r$-regular, the number of edges that are not mono-indexed in $C_n^r$ is $2r\frac{n-i}{r+1}$. Hence, the number of mono-indexed edges is $rn-2r\frac{n-i}{r+1}=\frac{r}{r+1}((r-1)n+2i)$.
\end{proof}

Figure \ref{G-kPCyc} illustrates the admissibility of weak IASIs by the squares of even and odd cycles.  

\begin{figure}[h!]
\centering
\includegraphics[scale=0.35]{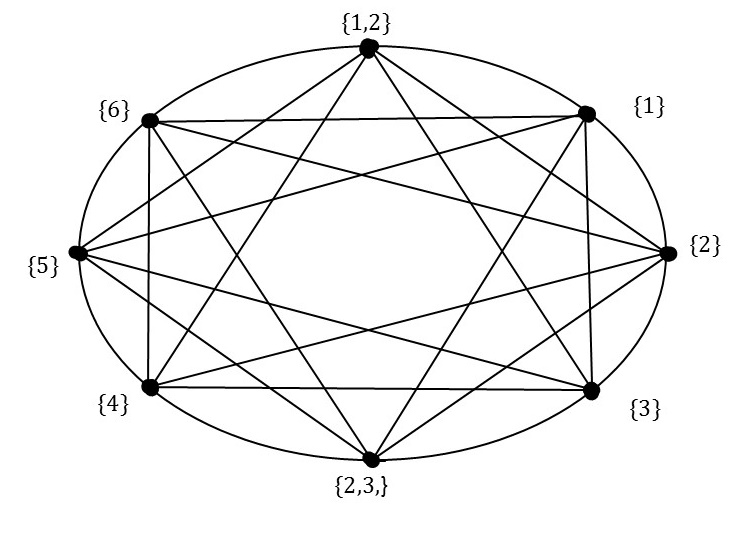}
\caption{Cube of a cycle with a weak IASI defined on it.}\label{G-kPCyc}
\end{figure}

\section{Conclusion}

In this paper, we have established some results on the admissibility of weak IASIs by certain graphs and graph powers. The admissibility of weak IASI by various graph classes, graph operations and graph products and finding the corresponding sparing numbers are still open. 

In this paper, we have not addressed the following problems, which are still open. The adjacency and incidence patterns of elements of the graph concerned will matter in determining its admissibility of weak IASI and the sparing number.

\begin{problem}{\rm
Find the sparing number of the $r$-th power of trees and in particular, binary trees for applicable values of $r$.}
\end{problem}

\begin{problem}{\rm
Find the sparing number of the $r$-th power of bipartite graph and in general, graphs that don't have a complete bipartite graphs as their subgraphs, for applicable values of $r$.}
\end{problem} 

\begin{problem}{\rm
Find the sparing number of the $r$-th power of an $n$-sun graph that is not complete, for applicable values of $r$.}
\end{problem}

\begin{problem}{\rm
Find the sparing number of the square of a split graph that is not complete.}
\end{problem}

Some other standard graph structures related to paths and cycles are lobster graph, ladder graphs, grid graphs and prism graphs. Hence, the following problems are also worth studying.

\begin{problem}{\rm
Find the sparing number of arbitrary powers of a lobster graph.}
\end{problem}

\begin{problem}{\rm
Find the sparing number of arbitrary powers of a ladder graphs $L_n$.}
\end{problem}

\begin{problem}{\rm
Find the sparing number of arbitrary powers of grid graphs (or lattice graphs) $L_{m,n}$.}
\end{problem}

\begin{problem}{\rm
Find the sparing number of arbitrary powers of prism graphs and anti-prism graphs.}
\end{problem}

\begin{problem}{\rm
Find the sparing number of arbitrary powers of armed crowns and dragon  graphs.}
\end{problem}

More properties and characteristics of different IASIs, both uniform and non-uniform, are yet to be investigated. The problems of establishing the necessary and sufficient conditions for various graphs and graph classes to have certain IASIs are also open.


\vspace{2cc}


\begin{thebibliography}{20}

\bibitem {A1} {\small {\sc B. D. Acharya}, {\it Set-Valuations and Their Applications}, MRI Lecture Notes in Applied Mathematics, No. 2, The Mehta Research Institute of Mathematics and Mathematical Physics, Allahabad, 1983.}

\bibitem {BLS} {\small {\sc A. Brandst\"{a}dt, V. B. Le and J. P. Spinrad}, {\it Graph Classes:A Survey}, SIAM, Philadelphia, 1999.}

\bibitem {BM1} {\small {\sc J. A. Bondy and U. S. R. Murty}, {\it Graph Theory}, Springer, 2008.}

\bibitem {JAG} {\small {\sc J. A. Gallian}, {\it A Dynamic Survey of Graph Labelling}, The Electronic Journal of Combinatorics (DS-6), 2013.}

\bibitem {GA} {\small {\sc K. A. Germina and T. M. K Anandavally}, {\it Integer Additive Set-Indexers of a Graph:Sum Square Graphs}, Journal of Combinatorics, Information and System Sciences, {\bf 37}(2-4)(2012), 345-358.}

\bibitem {GS1} {\small {\sc K. A. Germina and N. K. Sudev}, {\it On Weakly Uniform Integer Additive Set-Indexers of Graphs}, International Mathematical Forum {\bf 8}(37)(2013), 1827-34., DOI: 10.12988/imf.2013.310188.}

\bibitem {FH}  {\small {\sc F. Harary}, {\it Graph Theory}, Addison-Wesley Publishing Company Inc., 1994.}

\bibitem {GS0} {\small {\sc N K Sudev and K A Germina}, {\em On Integer Additive Set-Indexers of Graphs},Int. J. Math. Sci. \& Engg. Appl., {\bf 8}(2)(2014), 11-22.}

\bibitem {GS3} {\small {\sc N. K. Sudev and K. A. Germina}, {\em A Characterisation of Weak Integer Additive Set-Indexers of Graphs},Journal of Fuzzy Set Valued Analysis, {\bf 2014}(2014), Article Id:jfsva-00189, 1-7., DOI: 10.5899/2014/jfsva-00189.}

\bibitem {GS4} {\small {\sc N. K. Sudev and K. A. Germina}, {\it Weak Integer Additive Set-Indexers of Certain Graph Operations}, Global Journal of Mathematical Sciences: Theory and Practical, {\bf 6}(2)(2014), 25-36.}

\bibitem {GS5} {\small {\sc N. K. Sudev and K. A. Germina}, {\it A Note on the Sparing Number of Graphs}, Advances and Applications in Discrete Mathematics, {\bf 14}(1)(2014), 51-65.}

\bibitem {GS6} {\small {\sc N. K. Sudev and K. A. Germina}, {\it On the Sparing Number of Certain Graph Classes}, Journal of Discrete Mathematical Sciences nd Cryptography, {\bf 18}(1-2)(2015), 117-128, DOI: 10.1080/09720529.2014.962866.}

\bibitem {GS7} {\small {\sc N. K. Sudev and K. A. Germina}, {\it On the Sparing Number of Certain Graph Structures}, Annals of Pure and Applied Mathematics, {\bf 6}(2)(2014), 140-149.}

\bibitem{EWW} {\small {\sc E. W. Weisstein}, {\it CRC Concise Encyclopedia of Mathematics}, CRC press, 2011.}

\bibitem {DBW} {\small {\sc D. B. West}, {\it Introduction to Graph Theory}, Pearson Education Inc., 2001.}

\bibitem{GCI} Information System on Graph Classes and their Inclusions \url{www.graphclasses.org}

\end{thebibliography}
\end{document}